\documentclass[a4paper,12pt]{amsart}
	
\usepackage{fullpage}
\usepackage{amsfonts}
\usepackage{amssymb,amsthm,amsmath,color}
\usepackage{array}
\usepackage{mathrsfs}
\usepackage{dsfont}
\usepackage{lmodern}
\usepackage{bbm}
\usepackage{hyperref}
\usepackage{graphicx}
\usepackage[utf8]{inputenc}  
\usepackage[T1]{fontenc} 
\usepackage{stmaryrd}  
\usepackage{tikz}
\usepackage{tkz-euclide}
\usepackage{biblatex}
\numberwithin{equation}{section}

\DeclareMathAlphabet{\mathbbo}{U}{bbold}{m}{n}

\newtheorem{theorem}{Theorem}[section]
\newtheorem*{theoSansNum}{Theorem}
\newtheorem{lem}[theorem]{Lemma}
\newtheorem{defi}[theorem]{Definition}

\newtheorem{cor}[theorem]{Corollary}

\newenvironment{rmq}{\stepcounter{theorem}\noindent\textbf{Remark \thetheorem}.}{}
\newenvironment{ex}{\stepcounter{theorem}\noindent\textbf{Example \thetheorem.}}{}

\newcommand{\N}{\mathbb N}
\newcommand{\Z}{\mathbb Z}
\newcommand{\R}{\mathbb R}
\newcommand{\Q}{\mathbb Q}
\renewcommand{\mod}{\,\mathrm{mod}\,}

\begin{document}

\title{Large subsets avoiding algebraic patterns}

\author{A. Bailleul}
\address{ENS Paris-Saclay, Centre Borelli, UMR 9010, 91190 Gif-sur-Yvette, France}
\email{alexandre.bailleul@ens-paris-saclay.fr}

\author{R. Riblet}
\address{ENS Paris-Saclay, Centre Borelli, UMR 9010, 91190 Gif-sur-Yvette, France}
\email{robin.riblet@ens-paris-saclay.fr}

\date{}

  \begin{abstract}
We prove the existence of a subset of the torus with large sumsets and avoiding all linear patterns. This extends a result of Körner, who had shown that for any integer $q \geq 1$, there exists a subset $K$ of $\R/\Z$ satisfying no non-trivial linear relations of order $2q-1$ and such that $q.K$ has positive Lebesgue measure. Our method is based on transfinite induction, which also allows us to produce large sets in different senses (cardinality, outer Lebesgue measure or Hausdorff dimension) avoiding families of algebraic patterns, for example Sidon sets in infinite abelian groups with small $2$ and $3$-torsion or sets with no repeated distances in $\R^n$. We also discuss questions of measurability of such sets and the role of the axiom of choice in our constructions.
\end{abstract}

\maketitle

\section*{Introduction and main results}

Large subsets avoiding particular patterns constitute an important topic of interest in modern analysis, geometric measure theory and additive combinatorics. When a structural result exists, it is always interesting to see how far one can go to push two different constraints working against each other. For example, Roth's celebrated theorem \cite{Roth} states that every subset of the first $N$ integers devoid of arithmetic progressions of length 3 has cardinality $o(N)$ so such a set must be sparse. On the other hand it is natural to wonder if it is still possible to construct such a (sequence of) set(s) which is still somewhat dense, for example with a fast growing counting function. This was done by Salem-Spencer \cite{SS} (and later improved by Behrend \cite{Behr}), who showed one can construct such a set with an almost-linear growth. For another example, we can mention Ruzsa's construction \cite{Ruzsa}, which establishes the existence of an infinite Sidon set (see the discussion about Section \ref{PartieSidon} below) with growth $N^{\sqrt{2}-1+o(1)}$ in the natural numbers, while it is elementary that such a set cannot have growth larger than $N^{1/2}$ (in fact it cannot grow faster than $\frac{N^{1/2}}{\sqrt{\log N}}$ \cite{Stohr}). This is the best known result in this direction and it would certainly be very interesting to determine if this is optimal or not. In the continuous case, many papers are also dedicated to proving the existence of large sets avoiding certain patterns (see for example \cite{Korner, Mathe, CillerueloRuzsa, Keleti, Mattila, ErdosKakutani, Davies, Kunen}). The Kakeya problem, recently announced to be resolved in dimension $3$ \cite{WZ} can also be classified in this circle of ideas. In this context, the largeness conditions are often stated in terms of Lebesgue
measure, Hausdorff dimension, Fourier dimension, etc. In this paper, we develop a method based on transfinite induction (assuming a form of the axiom of choice when necessary) to extend several results in this direction.

In Section \ref{sec:lineq} we provide general terminology and results on linear equations in abelian groups that are used later in the paper.

In Section \ref{sec:indep}, we tackle our first construction of a large subset avoiding linear patterns. Recall the following celebrated result of Körner.

\begin{theoSansNum}[Körner, \cite{Korner}]
    If $q \geq 1$ is an integer, then we can find a closed subset $E$ of the torus $\R/\Z$ with the following properties.
    \begin{itemize}
        \item[(i)] $q.E$ has strictly positive Lebesgue measure.
        \item[(ii)] The equation $$\sum_{j=1}^{2q-1} m_j x_j = 0$$ has no non-trivial solution with $m_j \in \Z$ and the $x_j$ distinct points of $E$. 
    \end{itemize}
\end{theoSansNum}

In our construction, we generalize constraint (ii) above to having no non-trivial solution to any linear equation. We call such a set an \textit{independent set}.

\setcounter{section}{2}
\setcounter{theorem}{1}

\begin{theorem}\label{mainResult}
    There exists an independent subset $K$ of $\R/\Z$ such that $\lambda^*(K) =1$ (outer Lebesgue measure). In particular, $q.K$ has positive outer Lebesgue measure for any $q \geq 1$.
\end{theorem}

Section \ref{PartieSidon} deals with large Sidon subsets of abelian groups. In general, a Sidon set is a subset of an abelian semi-group such that its two-elements sums are essentially unique, \textit{i.e.} $a+b=c+d \Rightarrow \{a,b\} = \{c, d\}$. Sidon sets are extensively studied in the literature, and the main question regarding them is the 'maximum size' they can attain. The meaning of 'size' depends on the context. In the finite case, we are naturally concerned with the cardinality, and it is well known, following the works of Singer, Erd\H{o}s, Bose, and Chowla (see \cite{OBryant}), that the largest Sidon set $S$ contained within the first $n$ integers satisfies $|S| \sim \sqrt{n}$. In the positive integers, Ruzsa \cite{Ruzsa} proved the existence of an infinite Sidon set with growth $N^{\sqrt{2}-1+o(1)}$. It is also natural to ask how big the sumset of a Sidon set can be. Erd\H{o}s-Sarközi-S\'{o}s \cite{SidonPasBase2} showed that it is impossible for a Sidon set in $\N$ to be an asymptotic basis of order $2$, that is, such that $S+S$ contains all large enough integers. However, Pilatte \cite{Pilatte} proved the existence of a Sidon sequence which is also an asymptotic basis of order 3. In the non-discrete infinite case, the notion of 'size' can still have several different interpretations (see \cite{CillerueloRuzsa}, for example). It is easy to see that a measurable Sidon set in $\R$ has Lebesgue measure zero, but Körner's Theorem \ref{Kor} implies the existence of a Sidon subset $S$ of $\R/\Z$ such that $S + S + S$ has positive Lebesgue measure. Theorem \ref{Sidon} below vastly generalizes this, and in particular shows the existence of a Sidon subset $S$ of $\R/\Z$ such that $S + S = \mathbb{R}/\mathbb Z$. 

\setcounter{section}{3}
\setcounter{theorem}{0}

\begin{theorem} \label{Sidon}
Let $G$ be an infinite abelian group of cardinality $\kappa$ such that $|\ker(x \mapsto 2x)| < \kappa$ and $|\ker(x \mapsto 3x)| < \kappa$. Assume $E$ a subset of $G$ such that for every $x \in G$, $|\{(a, b) \in E^2 \mid a+b=x\}| = \kappa$. Then there exists a Sidon subset $S$ contained in $E$ such that $S+S=G$.
\end{theorem}

We also remark that Theorem \ref{inde} proves the existence of a Sidon set with Hausdorff dimension of 1. The existence of a Sidon set of Hausdorff dimension 1 was already known (see \cite{Keleti}) but we also discuss in the Section \ref{sec:compare} the lack of direct implication between full Hausdorff dimension and the measure of $S+S$. We prove the following two results.
\setcounter{section}{4}
\setcounter{theorem}{0}
\begin{theorem}
   There exists a set $E\subset[0,1]$ of Hausdorff dimension 1, such that $E+E$ have Lebesgue measure zero.
\end{theorem}

\begin{theorem}
    Let $h\in[1/2,1]$, then there exists a set $E$ of $[0,1]$ of Hausdorff dimension $h$ and such that $E+E$ has non-zero Lebesgue measure.
\end{theorem}

In Section \ref{ANLP}, we focus on the existence of sets that avoid certain geometric properties and have maximal Hausdorff dimension. First, we prove the following theorem. 

\setcounter{section}{5}
\setcounter{theorem}{1}
\begin{theorem}
    There exists a subset $A \subseteq \R$ which is algebraically independent and has full outer Lebesgue measure (and consequently Hausdorff dimension $1$).
\end{theorem}

\setcounter{theorem}{7}

Then, we focus on several geometric patterns that have been extensively studied in the literature (see \cite{Mathe}, and all of Section \ref{ANLP} for further discussion). In particular, our method is sufficiently flexible (thanks to Lemma \ref{brique}) to combine a lot of different constraints. As an example, we establish the following.

\begin{theorem}\label{thmIntro}
Let $n\in\N^*$. There exists a trapezoid-free subset $E\subset \R^n$ avoiding repeated distances that has total outer Lebesgue measure (and consequently Hausdorff dimension $n$) such that the points of $E$ do not form the same angle twice.
\end{theorem}

We end the paper with a proof that, in some sense, we need to use some form of the axiom of choice to show the existence of sets such as $K$ in Theorem \ref{mainResult}. This may appear as a very strong assumption, but we show in Lemma \ref{nonMesur} that such a set (or its sumset) has to be non-measurable, and thus some form of choice stronger than $\mathsf{DC}$ is required (see the discussion in Section \ref{NonMesur}).

\subsection*{Acknowledgments} We thank Anne de Roton for pointing out to us a mistake in an earlier version of this paper.

\setcounter{section}{0}

\section{Linear equations and independent sets}
\label{sec:lineq}

Let us first make precise the notion of non-trivial solutions to a linear equation.

\begin{defi}
Let $G$ be an abelian group, $A$ a subset of $G$, $n \geq 1$ an integer and $m_1, \dots, m_n \in \Z$. Let $(E)$ denote the linear equation $\sum_{i=1}^n m_i x_i = 0$. Its order is the number of non-zero integers $m_i$, unless they are all zero in which case we say the order is $0$. We say a solution $(x_1, \dots, x_n) \in A^n$ of $(E)$ is trivial when there exists a partition $(P_i)_{1 \leq i \leq k}$ of $\{1, \dots, n\}$  such that for $1 \leq i \leq k$ we have $j \mapsto x_j$ is constant on $P_i$ and $\displaystyle \sum_{j \in P_i} m_j x_j = 0$. Other $n$-tuples of elements of $A$ satisfying this equality are called non-trivial solutions. We say $A$ is independent when it has no non-trivial solution to any linear equation.
\end{defi}

In other words, trivial solutions correspond to solutions that are always present in a non-empty subset of $G$. It is easy to see that this kind of solutions exists if and only if $\displaystyle \sum_{j=1}^n m_j = 0$, and that $A$ is independent when its elements are linearly independent when we see $G$ as a $\Z$-module. We remark that this definition generalizes the one given by Ruzsa in \cite{Ruzsa2}.\\

\begin{ex}
\begin{enumerate}
    \item Any equation of the form $x_i-x_j=0$ can only admit trivial solutions.
    \item More interestingly, the equation $2x_1-x_2-x_3=0$ also has trivial solutions in any non-empty subset $A$ of $G$, namely with $x_1=x_2=x_3=a$ for any $a \in A$. Then $A$ admits no non-trivial solutions to this equation if and only if $A$ contains no arithmetic progression of length $3$. If it also admits no non-trivial solutions to the equation $x_1+x_2-x_3-x_4=0$, then $A$ is a called a Sidon set, an important notion in additive combinatorics (see section \ref{PartieSidon}).
    \item This definition also deals with torsion elements. For example if $x$ is a torsion element of order $m$, then $x$ is a trivial solution to the equation $mx=0$.
\end{enumerate}
\end{ex}


Independent sets appeared in works on harmonic analysis. For example, the celebrated Kronecker theorem states that a finite independent set of real numbers generates a dense subset of the torus (\cite[Theorem 2.2.4]{Queff}). See \cite[Chapter 5]{Rudin} and \cite{KS} for generalizations and more details on the links between independent sets and harmonic analysis.\\

\begin{defi}
Let $G$ be an abelian group, $A$ a proper subset of $G$, $n \geq 1$ an integer and $m_1, \dots, m_n \in \Z$. Assume $A$ has no non-trivial solution to the equation $$(\mathcal E) : \sum_{i=1}^n m_i x_i = 0.$$ We say a set $B \subset G \setminus A$ is admissible for $(\mathcal E)$ if $A \cup B$ still has no non-trivial solution to the equation.
\end{defi}

\begin{lem}\label{admiss}
Let $G$ be an abelian group, $A$ a proper subset of $G$, $n \geq 1$ an integer and $m_1, \dots, m_n \in \Z$. For every $k \in \Z \setminus \{0\}$, let $f_k$ be the endomorphism $x \mapsto kx$ of $G$. Assume $A$ has no non-trivial solution to the equation $$(\mathcal E) : \sum_{i=1}^n m_i x_i = 0.$$ Then the set of non-admissible elements $a \in G \setminus A$ has cardinality at most $$(|A|^{n-1} + 1) . \max \{|\ker f_k| \mid k \in \Z \setminus \{0\}\}.$$ 
\end{lem}

\begin{proof}
Assume $a \in G \setminus A$ is non-admissible for $(\mathcal{E})$. That means there exists a non-trivial solution $x_1, \dots, x_n \in A \cup \{a\}$ such that $\sum_{i=1}^n m_i x_i = 0$. Moreover, since $A$ has no non-trivial solution to $(\mathcal{E})$, at least one of the $x_i$'s is $a$ and letting $I = \{i \in \{1, \dots, n\} \mid x_i = a\}$ we have $I \neq \emptyset$, but also $I \subsetneq \{1, \dots, n\}$ since this is a non-trivial solution. Moreover, we have $k = \sum_{i \in I} m_i \neq 0$, otherwise, we would have a non-trivial solution to $(\mathcal E)$ with elements in $A$. Thus, either $ka= 0$ and $a \in \ker f_k$, or $ka = - \displaystyle \sum_{\underset{i \not \in I}{1 \leq i \leq n}} m_i x_i$ can take at most $|A|^{n-1}$ values and $a$ can assume at most $|A|^{n-1}.|\ker f_k|$ values.
\end{proof}

\section{A large independent set}
\label{sec:indep}

In \cite{Korner}, Körner proves the following theorem.

\begin{theorem}[Körner, \cite{Korner}] \label{Kor}
    If $q \geq 1$ is an integer, then we can find a closed subset $E$ of the torus $\R/\Z$ with the following properties.
    \begin{itemize}
        \item[(i)] $q.E$ has strictly positive Lebesgue measure.
        \item[(ii)] The equation $$\sum_{j=1}^{2q-1} m_j x_j = 0$$ has no non-trivial solution with $m_j \in \Z$ and the $x_j$ distinct points of $E$. 
    \end{itemize}
\end{theorem}

In other words, with Körner's construction, the higher the order of linear relations we want to avoid, the more summands we need \textit{a priori} to obtain positive measure. Our main result is the following.

\begin{theorem}\label{inde}
There exists an independent subset $K$ of $\R/\Z$ such that $\lambda^*(K)=1$. In particular $q.K$ has positive outer Lebesgue measure for any $q \geq 1$.
\end{theorem}

\begin{proof}
Let $(P_{\alpha})_{\alpha < \mathfrak c}$ be a well-ordering of the perfect closed subsets of $\R/\Z$, where $\mathfrak c$ denotes the cardinality of the continuum. There are clearly at least $\mathfrak c$ such sets because closed intervals are perfect, and there are at most $\mathfrak c$ closed subsets of $\R$. Let us construct by transfinite induction subsets $(K_{\alpha})_{\alpha < \mathfrak c}$ such that:
\begin{enumerate}
    \item $\forall \alpha, \beta < \mathfrak c$, $\alpha < \beta \Rightarrow K_{\alpha} \subset K_{\beta}$.
    \item $\forall \alpha < \mathfrak c$, $|K_{\alpha}| \leq |\alpha|+1 < \kappa$.
    \item $\forall \alpha < \mathfrak c$, $K_{\alpha} \cap P_{\alpha} \neq \emptyset$.
    \item $\forall \alpha < \mathfrak c$, $K_{\alpha}$ is independent.
\end{enumerate}
Take $x_0$ any element of $P_0$ and let $K_0 = \{x_0\}$. It obviously satisfies 2., 3. and 4. Now, let $\beta < \mathfrak c$ and assume $(K_{\alpha})_{\alpha < \beta}$ have been constructed satisfying the four properties above. Since $\bigcup_{\alpha < \beta} K_{\alpha}$ has cardinality at most $|\beta|.(|\beta|+1) < \mathfrak c$ and since there are $\aleph_0$ linear equations with integer coefficients, Lemma \ref{admiss} tells us there are $< \mathfrak c$ real numbers which are non-admissible for some linear equation with respect to $\bigcup_{\alpha < \beta} K_{\alpha}$. But $|P_{\beta}| = \mathfrak c$ (\cite[Theorem 2.8.3]{Kuc}), so we can find $x_{\beta} \in P_{\beta}$ such that $K_{\beta} = \bigcup_{\alpha < \beta} K_{\alpha} \cup \{x_{\beta}\}$ is independent, and it clearly satisfies all four conditions.

Let $K = \bigcup_{\alpha < \mathfrak c} K_{\alpha}$. Clearly, $K$ is an independent subset of $\R/\Z$. Moreover, $\lambda^*(K) = 1$. Indeed, if $C$ is a closed subset contained in $(\R/\Z) \setminus K$, then $C$ is countable, because any closed subset of $\R/\Z$ is a disjoint union of a perfect closed set and the countable set of its isolated points by the Cantor-Bendixson theorem (\cite[Theorem 2.8.2]{Kuc}). By definition, $(\R/\Z) \setminus K$ has inner Lebesgue measure $0$, hence $K$ has outer Lebesgue measure $1$.
\end{proof}

\begin{rmq}\label{RmqInde}
\begin{enumerate}
    \item The previous conclusion extends Körner's result \ref{Kor}, as clearly if $K$ has positive outer measure, then so does $q.K$ for any $q \geq 1$, and we also don't need to restrict the order of linear relations we are avoiding. However, in the proof of Theorem \ref{inde} we assume that $\R/\Z$ (or equivalently, $\R$) can be well-ordered, which notoriously requires some form of the axiom of choice. But Körner's result \ref{Kor} is obtained using the Baire Category Theorem, which, in its full generality, is equivalent to the axiom of dependent choice $\mathsf{DC}$ (\cite{Blair}). As we will show in the last section of this paper (see Lemma \ref{nonMesur}), a set $K$ satisfying the properties of Theorem \ref{inde} must be so that $K$ or $K+K$ is non-measurable. As shown by the celebrated works \cite{Sol} of Solovay (assuming the existence of an inaccessible cardinal), $\mathsf{ZF+DC}$ itself is not enough to prove the existence of such a set, so a stronger form of the axiom of choice is necessary to construct a set such as $K$.
    \item We remark that this set $K$ defines a subset $\tilde{K}$ of $\R$ such that $\tilde{K}$ has positive outer Lebesgue measure and $\tilde{K}$ only has trivial solutions to non-homogeneous linear equations with second term an integer $k \in \Z$, in the sense that $\displaystyle \sum_{i=1}^n m_i x_i = m \in \mathbb Z$ implies that there is a partition $(P_i)_{1 \leq i \leq k}$ of $\{1, \dots, n\}$ such that $j \mapsto x_j$ is constant on $P_i$ and $\displaystyle \sum_{j \in P_i} m_j x_j \in \mathbb Z$. Conversely any such set yields a subset of $\R/\Z$ such as our $K$. Our method of proof could be adapted to the setting over $\R$.
    \item It is noteworthy that it is impossible for a subset $E$ of $\R/\Z$ to be independent and such that $n.E=\R/\Z$ for some positive integer $n$. Indeed, such a set could be lifted to an independent set of reals with $n.\tilde E = \R$ as in the previous remark. This subset is then $\Q$-linearly independent but also a generating family, \textit{i.e.} be a Hamel basis. But if $e_1, \dots, e_{n+1}$ are $n+1$ distinct elements of $E$, then $e_1 + \dots + e_{n+1}$ should be a sum of $n$ elements of $E$, which contradicts their linear independence over $\Q$.
    \item A completely similar construction yields an independent subset of $\mathbb R \setminus \mathbb Z$ which is not meager in the sense of Baire. Instead of intersecting with every perfect sets, it suffices to intersect with every dense $G_{\delta}$, of which there are $\mathfrak c$ and they all contain a Cantor set, so $\mathfrak c$ elements.
\end{enumerate}
\end{rmq}

\section{Sidon subsets with full sumset}\label{PartieSidon}

In the case of relations of order $4$, we can obtain a stronger result, that can actually be extended to a large class of abelian groups. Recall that a Sidon set in an abelian semi-group is a subset $S$ such that for every $a,b,c,d \in S$, $a+b=c+d$ implies $\{a,b\} = \{c, d\}$. This is equivalent to saying that it has no non-trivial solution to the equation $x_1+x_2=x_3+x_4$.

\begin{theorem}
Let $G$ be an infinite abelian group of cardinality $\kappa$ such that $|\ker(x \mapsto 2x)| < \kappa$ and $|\ker(x \mapsto 3x)| < \kappa$. Assume $E$ a subset of $G$ such that for every $x \in G$, $|\{(a, b) \in E^2 \mid a+b=x\}| = \kappa$. Then there exists a Sidon subset $S$ contained in $E$ such that $S+S=G$.
\end{theorem}

\begin{proof}
The idea of the proof is the same as Theorem \ref{inde}.

Let $(x_{\alpha})_{\alpha < \kappa}$ be an enumeration of $G$. This allows us to construct by transfinite induction a family $(S_{\alpha})_{\alpha < \kappa}$ of subsets of $E$ satisfying the following properties:
\begin{enumerate}
    \item $x_{\alpha} \in S_{\alpha} + S_{\alpha}$.
    \item $S_{\alpha}$ is a Sidon set.
    \item $|S_{\alpha}| \leq 2(|\alpha|+1) < \kappa.$
    \item $S_{\alpha'}\subseteq S_\alpha$ for $\alpha'<\alpha$
\end{enumerate}

We start with $S_0 = \{a_0, b_0\}$ for some $a_0, b_0 \in E$ such that $a_0+b_0=x_0$. Assuming $S_{\alpha}$, for $\alpha < \beta$ have been constructed and setting $U_\beta=\bigcup_{\alpha < \beta} S_{\alpha}$, there are two cases to consider. If $x_{\beta} = a+b$ for some $a, b \in S_{\alpha}$ with $\alpha < \beta$ then we set $S_{\beta} = U_\beta$, which clearly satisfies the four required properties.

Otherwise, let $O$ be the set of obstructions
\begin{align*}
    O= & (U_\beta+U_\beta-U_\beta)\cup((U_\beta+U_\beta)/2)\cup (\left\lbrace x_\beta\right\rbrace+U_\beta-U_\beta-U_\beta)\cup ((2\left\lbrace x_\beta\right\rbrace-U_\beta-U_\beta)/2) \\ & \cup ((\left\lbrace x_\beta\right\rbrace+U_\beta-U_\beta)/2)\cup ((\left\lbrace x_\beta\right\rbrace+U_\beta)/3)\cup ((2\left\lbrace x_\beta\right\rbrace-U_\beta)/3)\cup (\left\lbrace x_\beta\right\rbrace /2),
\end{align*}
where for every $A \subseteq G$ and every integer $k$, $A/k$ denotes the (potentially empty) set of elements $g \in G$ such that $kg \in A$. By the hypothesis on the kernels of $x \mapsto 2x$ and $x \mapsto 3x$, and since $|U_{\beta}| < \kappa$, $|O| < \kappa$. Moreover, since $G$ is a group, then for any $o \in O$, there exists at most one $e \in E$ such that $x_\beta=o+e$, so we have 
$$|\{(a,b) \in ((O \cap E) \times E)\cup (E\times (O \cap E)) \mid a+b = x_{\beta}\}|\leqslant 2|U|<\kappa = |\{(a,b) \in E^2 \mid a+b = x_{\beta}\}|.$$
Therefore we can find $a_\beta, b_\beta \in E\setminus O$ such that $a_\beta+b_\beta=x_{\beta}$. Then we set $S_{\beta} = U_{\beta} \cup \{a_\beta, b_\beta\}$ which satisfies the four required properties. Indeed, items 1, 3 and 4 are obvious. To show item 2, assume by contradiction that $S_\beta$ is not a Sidon set. That is, suppose there exists a non-trivial solution $a+b=c+d$ in $S_\beta$. As this solution is non-trivial, $a_\beta$ and $b_\beta$ can only occur at most twice each among $a,b,c$ and $d$, and if one occurs twice, then it is on the same side of the equal sign. By symmetry, here are the different possible cases:
\begin{itemize}
    \item $a,b,c,d\in U_\beta$. But then there exists $\alpha<\beta$ such that $a,b,c,d\in S_\alpha$, which is impossible because $S_\alpha$ is a Sidon set.
    \item $a=a_\beta$, $b=b_\beta$, $c,d\in U_\beta$. But then $x_\beta=c+d$, which has been excluded in this case.
    \item all the other cases are impossible by definition of $O$, as explained in the table below.

    \begin{center}
       \begin{tabular}{ | c | c | }
        \hline 
   \textbf{Cases} & \textbf{Implication}  \\
   \hline  
   $a=a_\beta$ and $b,c,d\in U_\beta$ & $a_\beta\in (U_\beta+U_\beta-U_\beta )$ \\ \hline
   $a=b=a_\beta$ and $c,d\in U_\beta$ & $a_\beta\in ((U_\beta+U_\beta)/2)$ \\ \hline $a=b_\beta$ and $b,c,d\in U_\beta$ & $a_\beta\in (\left\lbrace x_\beta\right\rbrace+U_\beta-U_\beta-U_\beta)$ \\ \hline $a=b=b_\beta$ and $c,d\in U_\beta$ & $a_\beta\in((2\left\lbrace x_\beta\right\rbrace-U_\beta-U_\beta)/2)$ \\ \hline $a=a_\beta$, $c=b_\beta$ and $b,d\in U_\beta$ & $a_\beta\in ((\left\lbrace x_\beta\right\rbrace+U_\beta-U_\beta)/2)$ \\ \hline $a=b=a_\beta$, $c=b_\beta$ and $d\in U_\beta$ & $a_\beta\in ((\left\lbrace x_\beta\right\rbrace+U_\beta)/3)$ \\ \hline $a=a_\beta$, $c=d=b_\beta$ and $b\in U_\beta$ & $a_\beta\in ((2\left\lbrace x_\beta\right\rbrace-U_\beta)/3)$ \\ \hline $a=b=a_\beta$ and $c=d=b_\beta$ & $a_\beta\in (\left\lbrace x_\beta\right\rbrace /2)$ \\ \hline 
 \end{tabular}
    \end{center}
\end{itemize}
Therefore $S_\beta$ is a Sidon set. Finally, $S = \bigcup_{\alpha < \kappa} S_{\alpha}$ is a Sidon set included in $E$ such that $S+S=G$.
\end{proof}

\section{Comparing Hausdorff dimension and sumsets with measure}
\label{sec:compare}

Theorem \ref{Sidon} proves in particular that there exists a Sidon set $S\subset\R$ such that $S+S$ has non-zero Lebesgue measure, while Theorem \ref{mainResult} shows the existence of a Sidon set with Hausdorff dimension 1. It may be legitimate to ask whether there is a direct implication between these two theorems. In other words, is there any implication in $\R$ between $\dim_H(S)=1$ and $S+S$ having positive Lebesgue measure? The answer is no, as the following two theorems show.

\begin{theorem}\label{HGLP}
   There exists a set $E\subset[0,1]$ of Hausdorff dimension 1, such that $E+E$ have Lebesgue measure zero.
\end{theorem}

\begin{proof}
       Let $E$ be the subset of $[0,1]$ made up of reals written only with digits ranging from $0$ to $d-1$ in their expansion in the generalized basis $(2d)_{d\geqslant 2}=(4,6, \dots ,2n, \dots)$ :
       $$E=\left\lbrace \sum_{i\geq 1} \frac{a_i}{2^i (i+1)!} \ : \ a_i\in\{ 0,\dots,i\} \right\rbrace\subset [0,1].$$
       It is easily seen that $E+E$ is made up of real numbers written without the $d+1$-th digit $2d-1$ in the generalized basis $(2d)_{d\geqslant 2}$: 
       \begin{align*}
           E+E  &=\left\lbrace \sum_{i\geq 1} \frac{b_i}{2^i (i+1)!} \ : \ b_i\in\{ 0,\dots,2i\} \right\rbrace \\ & = [0,1] \setminus \left\lbrace \sum_{i\geq 1} \frac{c_i}{2^i (i+1)!} \ : \ c_i\in\{ 0,\dots,2i+1\} \ \text{and } \exists i_0\geq 1 \text{ such that } c_{i_0}=2i_0+1   \right\rbrace.
       \end{align*}
       
       We can see $E$ as a generalized Cantor set that can be constructed in stages, where each interval of stage $k$ is subdivided into $k+1$ intervals of length $\frac{1}{2(k+1)}$ times smaller.

       Here are the first steps in the construction.

       
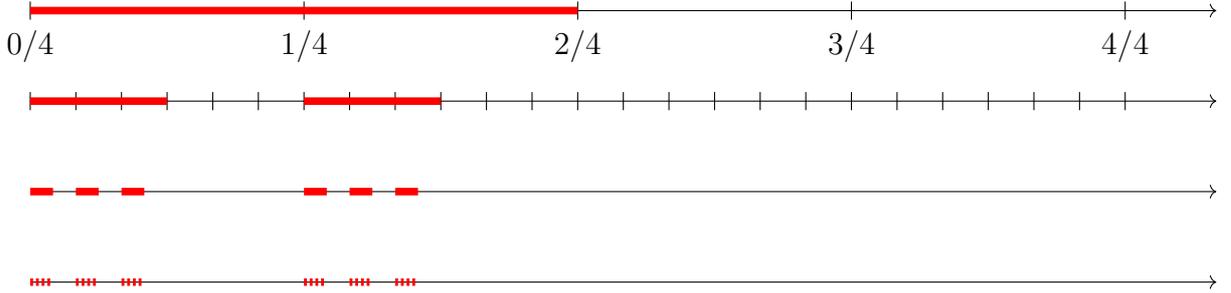
\begin{center}
   \begin{tikzpicture}[scale=1.2]
    \draw[->] (0,-1) -- (13,-1) node[anchor=north west] {};
        \foreach \y in {0,1/2,1,3,3.5,4}
    \foreach \x in {0,...,3}
    \draw[red, line width=1mm] (\x/16+\y,-1) -- (\x/16+1/32+\y,-1);   
    \draw[->] (0,0) -- (13,0) node[anchor=north west] {};
\draw[red, line width=1mm] (0,0) -- (1/4,0);
\draw[red, line width=1mm] (1/2,0) -- (3/4,0);
\draw[red, line width=1mm] (1,0) -- (5/4,0);
\draw[red, line width=1mm] (3,0) -- (13/4,0);
\draw[red, line width=1mm] (14/4,0) -- (15/4,0);
\draw[red, line width=1mm] (16/4,0) -- (17/4,0);
        \draw[->] (0,1) -- (13,1) node[anchor=north west] {};
    \foreach \x in {0,...,24}
        \draw (\x/2,1.1) -- (\x/2,0.9) node[anchor=north] {};
\draw[red, line width=1mm] (0,1) -- (6/4,1);
\draw[red, line width=1mm] (12/4,1) -- (18/4,1);
        \draw[->] (0,2) -- (13,2) node[anchor=north west] {};
    \foreach \x in {0,...,4}
        \draw (3*\x,2.1) -- (3*\x,1.9) node[anchor=north] {\x/4};
        \draw[red, line width=1mm] (0,2) -- (6,2);
\end{tikzpicture}
\end{center}

From Example 4.6 of \cite{Falco} (with $c=1/2$, $m_k=k$ and $\delta_k=\frac{1}{2^kk!}$), we have 
$$\dim_H(E)\geqslant \liminf_{k\rightarrow \infty}\frac{\log(m_1...m_k)}{-\log(\delta_k)}=\liminf_{k\rightarrow \infty}\frac{\log(k!)}{k\log(2)+\log(k!)}=1.$$

On the other hand, for any $k\geqslant 2$, let $I_k\subset \mathcal{P}([0,1])$ be the set of intervals $\left\lbrace \left[ \frac{i}{2^{k-1}k!} \mid \frac{i+1}{2^{k-1}k!}\right] \ \mid \ i\in\N\right\rbrace$. For any $k\geqslant2$, let's also set the following operators to construct $E+E$ algorithmically,
\begin{align*}
T_k \ : \mathcal{P}([0,1]) & \longrightarrow \mathcal{P}([0,1]) \\ A & \longmapsto \bigcup_{[a,b]\in I_k\cap \mathcal{P}(A)} \left[a,b-\frac{b-a}{2(k+1)}\right].
\end{align*}
This gives $E+E=\bigcap_{k\geqslant 2}A_k$, where $(A_k)_{k \geqslant 2}$ is the sequence defined by $A_2=[0,3/4]$ and $A_{k+1}=T_k(A_k)$. Here is a representation of $A_2$, $A_3$ and $A_4$ :

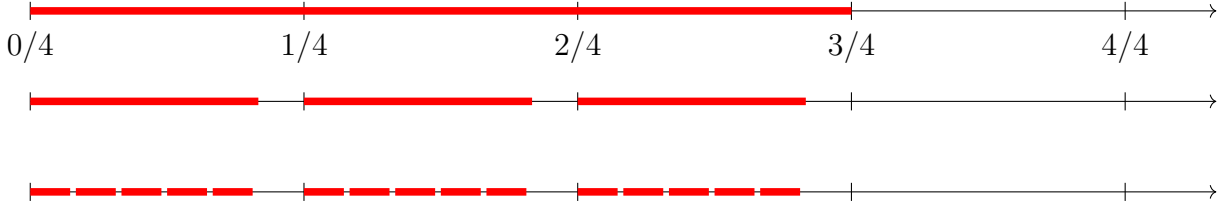
\begin{center}
  \begin{tikzpicture}[scale=1.2]
    \draw[->] (0,0) -- (13,0) node[anchor=north west] {};
    \foreach \x in {0,...,4}
        \draw (3*\x,0.1) -- (3*\x,-0.1) node[anchor=north] {};
    \foreach \y in {0,3,6} {
        \draw[red, line width=1mm] (0+\y,0) -- (105/240+\y,0);
        \draw[red, line width=1mm] (120/240+\y,0) -- (225/240+\y,0);
        \draw[red, line width=1mm] (1+\y,0) -- (345/240+\y,0);
        \draw[red, line width=1mm] (360/240+\y,0) -- (465/240+\y,0);
        \draw[red, line width=1mm] (480/240+\y,0) -- (585/240+\y,0);
    }
    \draw[->] (0,1) -- (13,1) node[anchor=north west] {};
    \foreach \x in {0,...,4}
        \draw (3*\x,1.1) -- (3*\x,.9) node[anchor=north] {};
    \draw[red, line width=1mm] (0,1) -- (15/6,1);
    \draw[red, line width=1mm] (12/4,1) -- (12/4+15/6,1);
    \draw[red, line width=1mm] (24/4,1) -- (24/4+15/6,1);
    \draw[->] (0,2) -- (13,2) node[anchor=north west] {};
    \foreach \x in {0,...,4}
        \draw (3*\x,2.1) -- (3*\x,1.9) node[anchor=north] {\x/4};
    \draw[red, line width=1mm] (0,2) -- (9,2);
\end{tikzpicture}
\end{center}
At step $k$ we only keep a proportion $\frac{2k-1}{2k}$ of the measure from the previous step, hence $\lambda(A_k)=\displaystyle\prod_{i=2}^k \frac{2i-1}{2i} = \displaystyle\prod_{i=2}^k \left(1 - \frac{1}{2i}\right)$ and finally
$$\lambda(E+E)= \lim_{k \to +\infty} \prod_{i=2}^k \left(1 - \frac{1}{2i}\right)=0,$$ since the series $\sum_{i \geq 2} \log \left(1 - \frac{1}{2i}\right)$ diverges to $-\infty$.
\end{proof}

Conversely, we prove the following result.

\begin{theorem}\label{HPLG}
    Let $h\in[1/2,1]$, then there exists a subset $E$ of $[0,1]$ of Hausdorff dimension $h$ and such that $E+E$ has non-zero Lebesgue measure.
\end{theorem}
\begin{proof}
    We only need to construct a set $E$ such that $\dim_H(E)\leq 1/2$ and $\lambda(E+E)>0$. For then any set $F$ with Hausdorff dimension $h>1/2$ containing $E$ (for example the union of $E$ and any set of Hausdorff dimension $h$) will satisfy the conditions of the Theorem.

    For any $n\in\N$, consider an additive $2$-base $B_n$ of $[0,n+1]$. Let $E\subset[0,1]$ be the set of elements whose digits in the generalized basis $(1/2,1/3, \dots ,1/n, \dots)$ are in $\prod_{n\in\N}B_n$. By the definition of a generalized basis, we have $E+E=[0,1]$, and we just need to make sure $E$ has the right dimension.
    
    To choose $B_n$ accordingly, we use a construction of Cassels (see \cite[Theorem 12, p.39]{HalberstamRoth}) that provides a $2$-base $\mathcal B$ of $\N$ such that its $n$-th element is $\beta n^2 + O(n)$, where $\beta > 0$ is a constant independent of $n$. We then truncate it up to $n$ to obtain $B_n = \mathcal{B} \cap [0, n]$. This satisfies $|B_n| = C \sqrt{n} + O(n^{1/4})$ for some constant $C > 0$ independent of $n$. Moreover, Cassels construction is such that this $2$-base contains at most $5$ consecutive integers. Thus, by approximating $E$ in steps (where the $n$-th step consists of fixing only the first $n$ digits in the above generalized basis), we can cover $E$ by at most $O(C^n\sqrt{n!})$ sets of sizes smaller than $\frac{5^n}{n!}$. Finally, we have (see \cite[Proposition 4.1, p.54]{Falco})
    $$\dim_H(E)\leq \lim_{n \to +\infty} \frac{\log\left( O(C^n\sqrt{n!})\right)}{- \log(5^n/n!)} \leq 1/2. $$
\end{proof} 
From this proof, we can easily deduce the following corollary.
\begin{cor}
    There exists a set $E\subset[0,1]$ of Hausdorff dimension $1/2$ such that $E+E=[0,1]$.
\end{cor}

Theorems \ref{HGLP} and \ref{HPLG} show that in $\R$, there is no direct implication between “Hausdorff dimension 1” and “sum set having non-zero Lebesgue measure”. 

\section{Avoiding non-linear patterns}\label{ANLP}

In this section, we apply the same method multiple times to address certain questions regarding classical studied topics. Most of the following proofs will be based on transfinite induction and the following key lemma.

\begin{lem}\label{brique}
    Let $n\in\N^*$ and $\mathcal{A}$, a family of zero sets of nonzero polynomials in $n$ variables such that $|\mathcal{A}|<\mathfrak c$. Then $\lambda_*^n(\bigcup_{A\in\mathcal{A}} A)=0$, where $\lambda^n_*$ denotes the $n$-dimensional inner Lebesgue measure.
\end{lem}

\begin{proof}
    We prove the result by induction on $n$. When $n=1$, elements of $\mathcal A$ are non-trivial algebraic closed sets on the line, i.e. are finite sets, and their union has cardinality less than the continuum, and thus cannot have positive inner measure by the Cantor-Bendixson theorem.
    Now, assume the result for all dimensions less than $n$ and assume by contradiction that we can find such a family with $\lambda_*^n(\bigcup_{A\in\mathcal{A}} A)>0$. By definition, there exists a closed set $F\subseteq \bigcup_{A\in\mathcal{A}} A$ such that $\lambda^n(F)>0$. Now, for all $A\in\mathcal{A}$, $A$ has a finite number of irreducible components. Since $|\mathcal{A}|<\mathfrak{c}$, we may assume $\mathcal{A}$ is composed of less than $\mathfrak{c}$ connected non-trivial algebraic closed sets (that is, different from the whole space). Each of these closed sets intersects a given hyperplane either in itself (if it is equal to it) or in a non-trivial subvariety of dimension at most $n-1$. As, $|\mathcal{A}|<\mathfrak{c}$, there exists a hyperplane $H$ intersecting each element of $\mathcal A$ as in the latter case. If $d$ is the normal direction to $H$, Fubini's theorem then yields
    $$0<\int_{\R^n}\mathbbm{1}_F(\mathbf{x})\text{d}\mathbf{x}=\int_d\int_{x+d^{\bot}}\mathbbm{1}_F(x,\mathbf{y})\text{d}\mathbf{y}\text{d} x,$$
    so there is $x\in d$ such that $\lambda_*^{n-1}(F\cap(x+d^{\bot}))>0$. But $F\cap(x+d^{\bot})$ is included in a union of less than $\mathfrak c$ non-trivial manifolds of the hyperplane $x+d^{\bot}$, so this is a contradiction with the induction hypothesis.

\end{proof}

\subsection{Algebraically independent subsets of $\R$}

First for $n=1$, we show there exists an algebraically independent subset of $\R$ with full outer Lebesgue measure.

\begin{theorem}
    There exists a subset $A \subseteq \R$ which is algebraically independent and has full outer Lebesgue measure (and consequently Hausdorff dimension $1$).
\end{theorem}

\begin{proof}
    Let $(P_\alpha)_{\alpha < \mathfrak c}$ be a transfinite enumeration of the perfect sets of non-zero measure in $\R$. We construct $A$ by transfinite induction, adding at step $\alpha$ an element $x_\alpha$ in $P_\alpha$ excluding the zero sets of all nonzero polynomials $Q(e_1,...,e_k,X)$ with coefficients in $\Q$, $k\in\N$ and $e_1,...,e_k$ varying within the set of previously constructed elements. This step is possible because there are fewer than $\mathfrak c$ such polynomials and $|P_\alpha |=\mathfrak c$.

    We can therefore successfully carry out the transfinite induction, and the set $A$ thus constructed intersects every perfect set of positive measure in $\R$. By the Cantor-Bendixson theorem, $\R \setminus A$ only contains closed sets of measure zero. Thus, $\lambda_*(\R \setminus A) = 0$. Therefore, we conclude that full outer Lebesgue measure, and hence $\dim_H(E) = 1$.
\end{proof}

\subsection{Subset without trapeze}

 We say that a set is trapezoid-free if it contains no three collinear points and no four distinct points that form two collinear vectors. M\'athé \cite{Mathe} proved the existence of a compact set $E\subset\R^n$ that is trapezoid-free and has Hausdorff dimension $n-1$.   This result is optimal in a certain sense, as it is known \cite{Mattila1995} that for any Borel set $A\subset \R^n$
  with $\dim_H(A)>n-1+s$, there exists a line $D$ such that $\dim_H(A\cap D)>s$. In particular, this implies that there are three collinear points in $A$, and therefore, $A$ cannot be trapezoid-free.

However, if we allow $A$ to be non-Borel, what is the maximum Hausdorff dimension it can attain? By assuming the axiom of choice, we provide an answer to this question.

\begin{theorem}\label{trap}
Let $n\in\N^*$. There exists a trapezoid-free subset $E\subset \R^n$ that has total outer Lebesgue measure (and consequently Hausdorff dimension $n$).
\end{theorem}

\begin{proof}
    Being trapezoid-free means that it is impossible to find four points defining two parallel sides. To achieve this, it suffices to construct points such that the first two coordinates form a basis of $\mathbb{R}^2$. We are thus led to consider the polynomial $Q(X,Y,Z,T) = (y_1 - x_1)(t_2 - z_2) - (y_2 - x_2)(t_1 - z_1)$ and seek to construct a set $E$ such that $Q$ does not vanish for any quadruple $(X,Y,Z,T) \in E^4$ with at least three of them distinct, $X \neq Y$ and $Z \neq T$. To do this, we follow the same method. Let $(P_\alpha)$ be a transfinite enumeration of the perfect sets of non-zero measure in $\R^n$. We construct $E$ by transfinite induction, adding at step $\alpha$ an element $X \in P_\alpha$ such that $Q(X,Y,Z,T) \neq 0$ and $X \neq Y$ for all $Y,Z \neq T$ in $E_{\alpha}$, the set of elements already constructed. Such an $X$ exists because $U_\alpha = \bigcup_{Y,Z,T \in E_\alpha, Z \neq T} \left\lbrace X \in \R^n \mid Q(X,Y,Z,T) = 0 \right\rbrace$ is a union of $|E_\alpha|^3$ hyperplanes, and since $|E_\alpha|^3 < \mathfrak{c}$, it follows that $U_\alpha$ has inner measure zero by Lemma \ref{brique}, and so has $U_{\alpha} \cup E_{\alpha}$. Therefore, $P_\alpha \setminus U_\alpha$ has non-zero outer measure and is consequently non-empty.

    The set $E$ as constructed contains no trapezoids and intersects every perfect set of positive measure in $\R^n$. By the Cantor-Bendixson theorem, $\R^n \setminus E$ only contains closed sets of measure zero. Thus, $\lambda_*(\R^n \setminus E) = 0$. Therefore, we conclude that full outer Lebesgue measure, and hence $\dim_H(E) = n$.
\end{proof}

\subsection{Subset without repeated angle}

M\'{a}thé \cite{Mathe} studied the Hausdorff dimension of sets in $\R^n$
  that do not contain a right triangle. He proved the existence of a compact subset of $\R^n$
  with Hausdorff dimension $n/2$ that does not contain any right triangle. Also he established the existence of a compact set in $\R^n$
  with Hausdorff dimension $n/8$, such that no triplet of its points forms an angle in a given countable family. By using the axiom of choice, we extend these two results by constructing a full Hausdorff dimension set in $\R^n$ that not contains the same angle twice.

\begin{theorem}\label{angle}
Let $n\in\N^*$. There exists a subset $E$ of $\R^n$ that has total outer Lebesgue measure (and consequently Hausdorff dimension $n$) such that the points of $E$ do not form the same angle twice.
\end{theorem}

\begin{proof}
   We shall once again employ a transfinite induction, but this time considering the following polynomials $Q_{\theta,Y,Z}(X)=\left(\sum_{i=1}^n (x_i-y_i)(x_i-z_i)\right)^2-\cos(\theta)^2\left\| X-Y \right\|^2\left\| X-Z \right\|^2$ where $Y$ and $Z$ belong to the set of previously constructed elements, and $\theta$ belongs to the set of angles formed by those points. These are all non-zero polynomials because
   $$Q_{\theta,Y,Z}(X)=\left( \cos \left(\widehat{X-Y,X-Z} \right)^2 -\cos(\theta)^2\right)\left\| X-Y \right\|^2\left\| X-Z \right\|^2 ,$$
   and for all fixed $Y,Z,\theta$, we can find $X$ such that $\cos \left(\widehat{X-Y,X-Z} \right)^2 \neq\cos(\theta)^2$, $\left\| X-Y \right\|\neq 0$ and $\left\| X-Z \right\|\neq 0$. At step $\alpha < \mathfrak c$, if we have constructed the set $E_\alpha$ and the set of angles they form is $A_\alpha$, we obtain that $\bigcup_{Y \in E_\alpha} \bigcup_{Z \in E_\alpha, Z \neq Y} \bigcup_{\theta \in A_\alpha} \left\lbrace X\in \R^n \vert Q_{\theta,Y,Z}(X)=0\right\rbrace$ has inner measure zero by Lemma \ref{brique} because at this stage of induction, there are strictly fewer than $\mathfrak{c}$ elements in $E_\alpha$ and $A_\alpha$. We can thus choose an element $X$ in the $\alpha$-th perfect set of $\R^n$ such that $Q_{\theta,Y,Z}(X)\neq 0$ for all $Y \neq Z \in E_\alpha$ and $\theta \in A_\alpha$, which implies that $(\widehat{X-Y,X-Z})$ is a different angle from the angles formed by th previously constructed points, thus completing the induction.
\end{proof}

\begin{rmq}
\begin{enumerate}
\item In addition, considering the polynomials $Q_{\theta,W,Y,Z}(X)=\left(\sum_{i=1}^n (x_i-y_i)(w_i-z_i)\right)^2-\cos(\theta)^2\left\| X-Y \right\|^2\left\| W-Z \right\|^2$ where $\theta$ belongs to the set of angles formed by two vectors of the already constructed elements, it is also possible to prevent the set $E$ from containing four directions (i.e., eight points) that form the same angle twice. 
\item By removing one point if necessary, this construction also yields a set devoid of right-angled triangle. This considerably extends the $n/2$ bound provided by M\'athé in the case of Borel sets. This was already established by Erd\H{o}s-Komj\'ath in \cite{ErdosKomjath}, assuming the Continuum Hypothesis. Indeed, they showed that the plane can be covered by a countable collection of sets devoid of right triangle. In particular, one of the elements of this partition is a subset of the plane with Hausdorff dimension 2 that does not contain any right triangle.
\item For any $s>0$, by possibly removing a set of points with Hausdorff dimension $n-s$ (which will not alter the Hausdorff dimension), this construction also provides a set without angles that belongs to a given family with Hausdorff dimension $n-s$. This new result both extends the $n/8$ M\'athé's bound for Borel sets and allows the exclusion of a significantly larger, non-countable family of angles.

\end{enumerate}
\end{rmq}

\subsection{Subset without repeated distance}

In a similar vein, it is compelling to investigate the maximum Hausdorff dimension of a set in $\mathbb{R}^n$ that avoids repeated distances.

The distance set conjecture ('Falconer distance problem' \cite{FalcoExp}) asserts that for every analytic set $E \subset \mathbb{R}^n$ $(n \geq 2)$ of Hausdorff dimension larger than $n / 2$, the set of distances formed by $E$,
$$
D(E)=\{|x-y|: x, y \in E\} \subset \mathbb{R}
$$
has positive Lebesgue measure. If $E \subset \mathbb{R}^n$ is an analytic set with $\operatorname{dim}_H E>n / 2$, we can find uncountably many disjoint Borel sets $E_i \subset E$ with $\operatorname{dim}_H E_i=\left(n / 2+\operatorname{dim}_H E\right) / 2$, and if the distance set conjecture holds, each $D\left(E_i\right)$ has positive Lebesgue measure. But it is impossible for uncountably many measurable sets of positive measure to be pairwise disjoint, and $E$ must have repeated distances. This means that if $E\subset \mathbb{R}^n$ is an analytic set avoiding repeated distances, then $\dim_H(E)\leqslant n/2$. Conversely, M\'athé \cite{Mathe} has established the existence of such a set with Hausdorff dimension $\max(1,n/2)$, which answers the question for analytic sets.

What if $E$ is not a analytic set? Erd\H{o}s-Kakutani \cite{ErdosKakutani}, Davies \cite{Davies} and Kunen \cite{Kunen} proved that for any $n \geqslant 1$, the Continuum Hypothesis (CH) is equivalent to the existence of a partition of $\mathbb{R}^n$ into a countably many pieces, with every distance occurring in every piece at most once (Erd\H{o}s and Kakutani for $n=1$, Davies for $n=2$, Kunen for $n\geqslant 3$). Thus, without the CH, such a partition is impossible, but does there always exist a set $E \subset \mathbb{R}^n$ with Hausdorff dimension $n$ that avoids repeated distances? We provide an answer to this question, assuming the axiom of choice.

\begin{theorem}\label{dist}
Let $n\in\N^*$. There exists a subset $E\subset \R^n$ avoiding repeated distances that has total outer Lebesgue measure (and consequently Hausdorff dimension $n$).
\end{theorem}

\begin{proof}
    This time, we consider the functions $Q_{W,Y,Z}(X)=\left\| X-W \right\|^2-\left\| Y-Z \right\|^2$ where $W,Y$ and $Z$ belong to the set of previously constructed elements. The equation $Q_{W,Y,Z}(X)=0$ represents a $n$-sphere centered at $W$ with a radius of $\left\| Y-Z \right\|$ and so by Lemma \ref{brique} we have
    $$\lambda_*^n\left( \bigcup_{W,Y,Z}\left\{ X\in\R^n | Q_{W,Y,Z}(X)=0\right\}\right)=0 ,$$ and the rest of the proof goes as in the previous Theorems.
\end{proof}

\begin{rmq}
\begin{enumerate}
    \item In particular, this construction provides a subset of $\mathbb{R}^n$
  with Hausdorff dimension $n$ that does not contain any isosceles triangles.
    \item By combining the constraint functions from the proofs of the three previous theorems, we establish the existence of a subset of $\mathbb{R}^n$ with Hausdorff dimension $n$ that is free of trapezoids, right triangles, and isosceles triangles. In fact, this proves the following theorem.
\end{enumerate}
\end{rmq}

\begin{theorem}
Let $n\in\N^*$. There exists a trapezoid-free subset $E\subset \R^n$ avoiding repeated distances that has total outer Lebesgue measure (and consequently Hausdorff dimension $n$) such that the points of $E$ do not form the same angle twice.
\end{theorem}

\section{A non-measurable set and the role of the axiom of choice}\label{NonMesur}

We now proceed to show that, if $K$ is the set in Theorem \ref{inde}, then either $K$ or $K+K$ is not Lebesgue-measurable. As discussed in Remark \ref{RmqInde}, this implies that $K$ cannot be shown to exist using only $\mathsf{ZF}+\mathsf{DC}$. The following lemma is folklore.

\begin{lem}\label{liaison}
Let $E \subset \R$ be measurable with positive Lebesgue measure. Then there exist $m, n\in\Z\setminus\{ 0\}$ and a non-trivial solution in $E$ to the equation $mx+ny=0$.
\end{lem}

\begin{proof}
Let $\varepsilon=\lambda(E)>0$. By the regularity of Lebesgue measure, there exists a compact $K\subset E$ and an open set $O$ containing $E$ such that $\lambda(K)>2\varepsilon/3$ and $\lambda(O)<4\varepsilon/3$. Since $O$ is open and contains $K$, $\R\setminus O$ is a closed set disjoint from $K$ and thus $d(\R\setminus O,K)=\alpha>0$.

Let $U=\left]1-\alpha/(2\max \vert K\vert),1+\alpha/(2\max \vert K\vert) \right[$. Then $UK \subset O$. Indeed, otherwise, there would be $o\in \R\setminus O$ and $k\in K$ such that $o=k(1+\beta)$, where $\beta\in U-1$. But $o-k\in \R\setminus \left]-\alpha,\alpha\right[$ and $\beta k\in \left[ -\alpha/2,\alpha/2\right]$, which is absurd.

Thus $UK\subset O$ and if $u\in U$, we have $$\lambda(K\cap uK)=\lambda(K)+\lambda(uK)-\lambda(K\cup UK)\geqslant 2\lambda(K)-\lambda(O)>0,$$ therefore $K\cap uK$ is non-empty, and $u\in KK^{-1}$. Therefore, $U\subset KK^{-1}$. Finally, $U\setminus\left\lbrace 0,1 \right\rbrace$ contains a rational $m/n\in KK^{-1}$ which means there exist $k,k'\in K$ such that $m/n=k/k'$, \textit{i.e.} $mk-nk'=0$. As $m/n\neq 0,1$, $m,n\neq 0$ and $k\neq k'$, so this is indeed a non-trivial solution.
\end{proof}

\begin{rmq}
    We note that this is a bit stronger than Steinhaus's lemma (\cite[Theorem 3.7.1]{Kuc}), stating that $E+E$ has non-empty interior. It is less known that there exists a Baire category version of this result, see for example \cite[Theorem 2.9.1]{Kuc}.
\end{rmq}

\begin{lem}\label{nonMesur}
    Let $E$ be a subset of $\R/\Z$ such that $E+E$ has positive outer Lebesgue measure, and such that $E$ has no non-trivial solutions to linear equations of order $4$. Then at least one among $K$ and $K+K$ is not Lebesgue-measurable.
\end{lem}

\begin{proof}
    If $E$ and $E+E$ were measurable, then we would have $\lambda(E+E)=\lambda^*(E+E)>0$. But 
    $$E+E\subset \left((E\setminus(\Q/\Z))+(E\setminus(\Q/\Z))\right)\cup (E+(\Q/\Z))\cup ((\Q/\Z)+(\Q/\Z)),$$
    so one of the three has to have positive Lebesgue measure. Since $\Q/\Z$ is countable and Lebesgue measure is invariant by translation, we obtain $\lambda((E\setminus(\Q/\Z))+(E\setminus(\Q/\Z)))>0$ or $\lambda(E) > 0$, which implies the former. By lifting to $[0, 1)$ and using Lemma \ref{liaison}, we find $m,n\in\Z$ and $x, y\in [0, 1)$ such that $x + \mathbb Z, y + \mathbb Z \in (E \setminus (\mathbb Q/\mathbb Z)) + (E \setminus (\mathbb Q/\mathbb Z))$ such that $mx+ny=0$ and $x \neq y$. Thus, there exist $x_1,x_2,y_1,y_2\in E\setminus(\Q/\Z)$ such that $x_1+x_2=x\mod 1$ and $y_1+y_2=y\mod 1$, but $x\neq y$ so $x_1+x_2\neq y_1+y_2$ and so $\left\lbrace x_1,x_2 \right\rbrace \neq \left\lbrace y_1,y_2 \right\rbrace$ in $\R/\Z$. Moreover as $x_1,x_2,y_1,y_2\notin \Q/\Z$, we have $mx_1,mx_2,ny_1,ny_2\neq 0$ and so  $(x_1,x_2,y_1,y_2)$ is a non-trivial solution in $E$ to the equation $mx+my+nz+nt=0$, which is absurd.
\end{proof}

\begin{rmq}
We cannot conclude that $E$ itself is not measurable because there exist measurable sets $E$ such that $E+E$ is not measurable, as was proved in \cite{Rub}, building on a previous result of Sierpiński, who showed that the sum of two measurable sets need not be measurable.
\end{rmq}

\begin{cor} \label{Indep}
Assuming the consistency of $\mathsf{ZF} +$"there exists an inaccessible cardinal", $\mathsf{ZF}+\mathsf{DC}$ cannot prove the existence of $K$ in Theorem \ref{inde}.
\end{cor}

\begin{proof}
    Solovay's model (\cite{Sol}) satisfies $\mathsf{ZF+DC}$ and every subset of $\R$, hence $\R / \Z$ is Lebesgue measurable. Therefore, the set $K$ cannot be constructed assuming only $\mathsf{ZF+DC}$.
\end{proof}

We remark that avoiding linear relations of order at least $4$ corresponds to $q=3$ in Körner's Theorem \ref{Kor}. This raises the following questions.\\

\noindent \textbf{Question 1.} Does there exist a measurable subset $E$ of $\R/\Z$ such that $E$ has no non-trivial linear relations of order at most $3$ and such that $\lambda(E+E) > 0$?\\

For related works, we remark that it is known that $\mathsf{ZF} + \mathsf{DC}$ cannot prove the existence of a Hamel basis, and that this existence doesn't imply that $\R$ is well-ordered (see \cite{Hamel}).\\

\noindent \textbf{Question 2.} Is it possible to construct a set such as $K$ without well-ordering the real line?\\

Finally, Shelah showed (\cite{Shelah}) that Solovay's result must use the consistency of $\mathsf{ZF} +$"there exists an inaccessible cardinal".\\

\noindent \textbf{Question 3.} Can we prove that $\mathsf{ZF}+\mathsf{DC}$ cannot prove the existence of $K$ only assuming the consistency of $\mathsf{ZF}$?\\

Shelah also showed in \cite{Shelah} that one doesn't need the consistency of an inaccessible cardinal to obtain a model $\mathsf{ZF} + \mathsf{DC}$ in which every subset of $\R$ has the Baire property. Showing that a set such as $K$ cannot have the Baire property would answer the above question in the affirmative.

Going back to the construction of Sidon sets with large sumset in Section \ref{PartieSidon}, we remark that in the case of $\R$, our construction can be made inside a null set $E$ (for example a set of the form $C+\Z$ for some convenient Cantor set in $[0, 1]$), so a Sidon set $S$ such that $S+S=\R$ can indeed be measurable, and therefore we cannot prove a result similar to Corollary \ref{Indep}. The same argument shows that $S$ can have the Baire property. In general, a measurable Sidon set $S$ such that $S+S=\R$ is a null set (argue as in Lemma \ref{nonMesur}).  Similarly, a Sidon set $S$ with the Baire property and such that $S+S=\R$ is of first category since a non-meager subset of $\R$ with the Baire property contains arbitrarily large arithmetic progressions (\cite{Bosh}). This raises the following question.\\

\noindent \textbf{Question 4.} Is it possible to construct a Sidon subset $S$ of $\R$ such that $S+S=\R$ in $\mathsf{ZF+DC}$?\\

The existence of a Hamel basis of $\R$ over $\Q$ implies the existence of a Vitali set (\cite[Lemma 1.1]{Hamel}), which doesn't have the Baire property and so, as before, it cannot be established in $\mathsf{ZF+DC}$. Showing that we can produce a Vitali set from a Sidon set with sum $\R$ would answer this question negatively.

\printbibliography

\end{document}